\documentclass[12pt]{amsart}   
\usepackage{amsmath,amsthm,amssymb,mathtools,marvosym}
\usepackage[english]{babel}
\usepackage[autostyle]{csquotes}
\usepackage[autostyle]{csquotes}
\newtheorem{theorem}{Theorem}[section]
\newtheorem{Lemma}[theorem]{Lemma}
\newtheorem{Proposition}[theorem]{Proposition}
\newtheorem{Corollary}[theorem]{Corollary}
\newtheorem{Definition}[theorem]{Definition}
\newtheorem{Example}[theorem]{Example}

\begin{document}
\title[TOPOLOGICALLY STABLE EQUICONTINUOUS NON-AUTONOMOUS SYSTEMS]{TOPOLOGICALLY STABLE EQUICONTINUOUS NON-AUTONOMOUS SYSTEMS} 
\author[Abdul Gaffar Khan, Pramod Kumar Das, Tarun Das]{Abdul Gaffar Khan$^{1}$, Pramod Kumar Das$^{2}$ and Tarun Das$^{1}$}                 
\keywords{Expansivity, Shadowing, Topological Stability \vspace*{0.08cm}\\ 
\vspace*{0.01cm}
\Letter{Tarun Das} \\
\vspace*{0.08cm}
tarukd@gmail.com \\
\vspace*{0.01cm}
Abdul Gaffar Khan \\
\vspace*{0.08cm}
gaffarkhan18@gmail.com\\
\vspace*{0.01cm}
Pramod Kumar Das\\
\vspace*{0.08cm}
pramodkumar.das@nmims.edu\\
\vspace*{0.01cm}
\textit{$^{1}$Department of Mathematics, Faculty of Mathematical Sciences, University of Delhi, Delhi, India.}\\ 
\hspace*{0.11cm}\textit{$^{2}$School of Mathematical Sciences, Narsee Monjee Institute of Management Studies, Vile Parle, Mumbai-400056,  India.} 
}

\begin{abstract}
We find sufficient conditions for commutative non-autonomous systems on certain metric spaces to be topologically stable. In particular, we prove that (i) Every mean equicontinuous, mean expansive system with strong average shadowing property is topologically stable. (ii) Every equicontinuous, recurrently expansive system with almost shadowing property is topologically stable. (iii) Every equicontinuous, expansive system with shadowing property is topologically stable.
\end{abstract}
\maketitle

\section{Introduction} 

In experiments, it is seldom possible to measure a physical quantity without any error. Therefore, only those properties that are unchanged under small perturbations are physically relevant. In topological dynamics, a meaningful way to perturb a system is through a morphism (continuous map).  
\medskip

A homeomorphism (resp. continuous map) $f$ on a metric space $X$ is said to be topologically stable if for every $\epsilon>0$ there exists $\delta>0$ such that if $h$ is another homeomorphism (resp. continuous map) on $X$ satisfying $d(f(x),h(x))<\delta$ for all $x\in X$ then there is a continuous map $k:X\rightarrow X$ satisfying $f\circ k = k\circ h$ and $d(k(x),x)<\epsilon$ for all $x\in X$.        
\medskip 

This particular concept of stability is popularly known as topological stability which was originally introduced \cite{W} for a diffeomorphism on compact smooth manifold. By looking at the significance of the concept, it is worth to identify those dynamical properties which imply topological stability. One of such result in topological dynamics is ``Walters stability theorem" which states that expansive homeomorphisms with shadowing property on compact metric spaces are topologically stable \cite{WO}.            
\medskip

The notion of expansivity expresses the worse case unpredictability of a system. Although such unpredictable behaviour of symbolic flows was recognized earlier, the concept of expansivity for homeomorphisms on general metric spaces was introduced \cite{UU} in the middle of the twentieth century. The expansive behaviour of continuous maps is popularly known as positive expansivity \cite{CKE}. 
\medskip

For a continuous map $f$ on a metric space $X$ and fixed $x_0\in X$, identifying those $x\in X$ whose orbit follow that of $x_0$ for a long time and hence, understanding the asymptotic behaviour of $f^n(x)$ relative to $f^n(x_0)$ can provide deep insight of the system. Anosov closing lemma \cite{A} provides us with such information for a differentiable map on compact smooth manifold. Although the terminology ``shadowing" was originated from this lemma in differentiable dynamics, the notion of shadowing property has played a central role in topological dynamics. The idea behind the notion of shadowing is to guarantee the existence of an actual orbit with a particular behaviour by giving evidence of the existence of a pseudo orbit with the same behaviour. For general qualitative study on shadowing property, one may refer to \cite{AHT}.   
\medskip

In \cite{TDT}, authors studied notions of expansivity and shadowing property in the context of non-autonomous systems and proved ``Walters stability theorem" in this settings. Such systems occur as mathematical models of real life problems affected by two or more distinct external forces in different time span. These systems appear in various branches like informatics, quantum mechanics, biology etcetera. The earliest known example with connection to biology arise while solving the famous ``mathematical rabbit problem" which was appeared in the book ``Liber Abaci" written by Fibonacci in the year 1202. This problem can be written in the form of second order difference equation \cite{CD} which is also known as Fibonacci sequence and in this representation every state depends explicitly on the current time and therefore, it can be seen from the context of non-autonomous system. Because of such frequent occurrence of non-autonomous systems in practical problems, it is important to know about stability of such systems. 
\medskip

In this paper, we prove the following results which provide sufficient conditions for non-autonomous systems to be topologically stable.                 
\medskip

Let $F$ be a commutative non-autonomous system on a Mandelkern locally compact metric space.
\\
(i) If $F$ is mean equicontinuous, mean expansive and has strong average shadowing property, then it is topologically stable.    
\\
(ii) If $F$ is equicontinuous, recurrently expansive and has almost shadowing property, then it is topologically stable.
\\
(iii) If $F$ is equicontinuous, expansive and has shadowing property, then it is topologically stable.  
\medskip

The first part of the above result shows that average shadowing property is not only useful in the investigation of chaos \cite{BD} but also in the investigation of stability. The study of such useful property of a system was also initiated in the context of flows \cite{GSX}, iterated function systems \cite{B} and non-autonomous systems \cite{MRT}. The second part shows the connection of topological stability with an weaker form of shadowing property called almost shadowing property which is useful when one is not interested in the initial behaviour of the system. The final part improves the conclusion of Theorem 4.1 \cite{TDT} under stronger hypothesis.  
\medskip

Before proving these results we introduce the above mentioned notions and study their general properties. 

\section{Definitions and General Properties}

Throughout this paper $\mathbb{Z}$, $\mathbb{N}$ and $\mathbb{N}^+$ denote the set of all integers, the set of all non-negative integers and the set of all positive integers respectively. A pair $(X,d)$ denotes a metric space with metric $d$ and if no confusion arises of the concerned metric $d$, then we simply write $X$ is a metric space. All maps between metric spaces are assumed to be uniformly continuous. We say that the collection $F=\lbrace f_i:X\rightarrow X\rbrace_{i \in \mathbb{N}^{+}}$ is a non-autonomous system (NAS) on $X$ generated by the sequence of maps $\lbrace f_{i}\rbrace_{i\in\mathbb{N}^+}$. $F$ is said to be autonomous system generated by $f$, if $f_{i} = f$ for all $i \in \mathbb{N}^{+}$ and in this case, we simply write $F=\langle f \rangle$. We say that $F$ is periodic with period $m \in \mathbb{N}^{+}$ if $f_{mi+j} = f_{j}$ for all $i,j \in \mathbb{N}^{+}$. $F$ is said to be commutative if $f_{i}\circ f_{j} = f_{j}\circ f_{i}$ for all $i, j \in \mathbb{N}^{+}$. We say that $F$ is surjective if $f_{j}$ is surjective for all $j\in \mathbb{N}^{+}$. The set of all NAS on $X$, the set of all NAS of period $m$ on $X$ and the set of all commutative NAS on $X$ are denoted by $N(X)$, $N_{m}(X)$ and $NC(X)$ respectively.  
\medskip

Let $(X,d)$ and $(Y,p)$ be metric spaces. The product of $F\in N(X)$ and $G\in N(Y)$ is defined as $F\times G = \lbrace f_{i}\times g_{i} : X\times Y \rightarrow X\times Y\rbrace_{i \in \mathbb{N}^{+}}$, where $X\times Y$ is equipped with the metric $q((x_{1}, y_{1}), (x_{2}, y_{2})) =$ max $\lbrace d(x_{1}, x_{2}), p(y_{1}, y_{2})\rbrace$. For $F\in N(X)$ and $n\in\mathbb{N}$, we denote $F_n = f_{n}\circ f_{n-1}\circ . . .  f_{1}\circ f_{0}$, where $f_{0}$ denotes the identity map on $X$. For any $j \leq k$, we define $F_{[j, k]} = f_k\circ f_{k-1}\circ . . .\circ f_{j+1}\circ f_j$. For any $k \in \mathbb{N}^{+}$, the $k^{th}$-iterate of $F$ is given by $F^k = \lbrace F_{[(i-1)k+1,ik]}\rbrace_{i \in \mathbb{N}^{+}}$.  
\medskip

$F\in N(X)$ is said to be equicontinuous if the family $\lbrace f_i\rbrace_{i\in\mathbb{N}^{+}}$ is equicontinuous i.e. for every $\epsilon > 0$ there exists $\delta > 0$ such that $d(x,y) < \delta$ implies $d(f_{i}(x), f_{i}(y)) < \epsilon$ for all $i \in \mathbb{N}^{+}$. 
\medskip

We now introduce mean equicontinuous NAS and give examples satisfying the same.    

\begin{Definition}
$F\in N(X)$ is said to be mean equicontinuous (ME) if for every $\epsilon > 0$ there exists $\delta > 0$ such that if a pair of sequences $\lbrace x_{i}\rbrace_{i=0}^{\infty}$ and $\lbrace y_{i}\rbrace_{i=0}^{\infty}$ satisfy $\frac{1}{n}\sum_{i=0}^{n-1}d(x_{i}, y_{i})<\delta$, then $\frac{1}{n}\sum_{i=0}^{n-1}d(f_{j}(x_{i}), f_{j}(y_{i})) < \epsilon$ for all $j\in\mathbb{N}^{+}$. A continuous map $f$ is said to be mean continuous (MC) if $F=\langle f\rangle$ is ME. A homeomorphism $f$ is said to be mean equivalence (MEQ), if both $f$ and $f^{-1}$ are MC. 
\label{D2.1}
\end{Definition} 

We say that $F\in N(X)$ and $G\in N(Y)$ are uniformly conjugate, if there exists a uniform equivalence $h : Y\rightarrow X$ such that $f_{n}\circ h = h\circ g_{n}$ for all $n\in \mathbb{N}$. In addition, if $h$ is MEQ then we say that $F$ and $G$ are mean conjugate. We say that a property of a NAS is a uniform dynamical (resp. mean dynamical) property if it is preserved under uniform (resp.  mean) conjugacy. One can easily check that, every mean continuous map is uniformly continuous. Hence every uniform dynamical property is a mean dynamical property. 

\begin{Example}
(i) An isometry is MEQ and a contraction map is MC. 
\\
(ii) The tent map $f:[0,1]\rightarrow [0,1]$ given by $f(x) = 2 $min$\lbrace x, 1-x\rbrace$ is MC.
\\
(iii) Let $X =\prod_{i\in \mathbb{Z}} X_{i}$ be equipped with the metric $d(x,y)=\sum_{i=-\infty}^{\infty}\frac{|x_{i} - y_{i}|}{2^{|i|}}$, where $X_{i} = \lbrace 0, 1\rbrace$. Let $f : X\rightarrow X$ be given by $f(x) = y$, where $y_{i} = x_{i+1}$ for all $i\in \mathbb{Z}$. Then, $f$ is MEQ. Further, since for all $x,y\in X$, we have $d(f(x), f(y)) \leq 2d(x, y)$ and $d(f^{-1}(x), f^{-1}(y)) \leq 2d(x, y)$, therefore the system $F=\lbrace f, f^{-1}, f, f^{-1}, \underbrace{f^{-1}, f^{-1}}_\text{$2$-times}, \underbrace{f, f}_\text{$2$-times}, 
\\
f, f^{-1},\underbrace{f^{-1}, f^{-1}}_\text{$2$-times},\underbrace{f, f}_\text{$2$-times}, \underbrace{f, f, f}_\text{$3$-times},\underbrace{f^{-1}, f^{-1}, f^{-1}}_\text{$3$-times},$ $\underbrace{f^{-1}, f^{-1}, f^{-1}, f^{-1}}_\text{$4$-times},\underbrace{f, f, f, f}_\text{$4$-times}...  \rbrace$ is ME. 
\end{Example}

\begin{Proposition}
If $F\in N(X)$ and $G\in N(Y)$, then following statements hold. 
\begin{enumerate}
\item[(a)] $F$ and $G$ are ME if and only if $F\times G$ is ME.
\item[(b)] If $F$ is ME, then $F^{k}$ is ME for all $k\in \mathbb{N}^{+}$.
\end{enumerate}
\label{P3.2}
\end{Proposition}
\begin{proof}
For the proof of (a), use $p,q \leq max\lbrace p, q\rbrace \leq p + q$ for all $p, q \geq 0$. Proof of (b) follows from the definition.   
\end{proof} 

$F\in N(X)$ is said to be expansive \cite{TDT} with expansive constant $0 < \mathfrak{c} < 1$ if for each pair of distinct points $x, y\in X$, there exists $n\in \mathbb{N}$ such that $d(F_{n}(x), F_{n}(y)) > \mathfrak{c}$. In literature, if $F=\langle f\rangle$ is expansive, then $f$ is said to be positively expansive. Recall from \cite{CKE} that, if there exists a continuous injective map $f$ on compact $X$ such that $F=\langle f\rangle$ is expansive, then $X$ is finite. The following examples justify that this is not true for non-autonomous systems. 

\begin{Example}
(i) Let $X =\lbrace \frac{1}{m}, 1-\frac{1}{m}:m\in \mathbb{N}^{+}\rbrace$ and $f : X\rightarrow X$ be given by $f(0) = 0$, $f(1) = 1$ and $f(x) = x^{+}$ otherwise, where $x^+$ is the immediate right to $x$. Then, one can check that $F = \lbrace f, f^{-1}, f^{-2}, f^{2}, f^{3}, f^{-3}, f^{-4}, f^{4}, . . . \rbrace$ is expansive with expansivity constant $ 0 < \alpha < \frac{1}{6}$.
\\
(ii) Let $X =\prod_{i\in \mathbb{Z}} X_{i}$ be equipped with the metric $d(x,y)=\sum_{i=-\infty}^{\infty}\frac{|x_{i} - y_{i}|}{2^{|i|}}$, where $X_{i} = \lbrace 0, 1\rbrace$. Let $f : X\rightarrow X$ be given by $f(x) = y$, where $y_{i} = x_{i+1}$ for all $i\in \mathbb{Z}$. One can check that systems $F=\lbrace f, f^{-1}, f, f^{-1}, \underbrace{f^{-1}, f^{-1}}_\text{$2$-times}, \underbrace{f, f}_\text{$2$-times}, f, f^{-1}, \underbrace{f^{-1}, f^{-1}}_\text{$2$-times}, \underbrace{f, f}_\text{$2$-times}, \underbrace{f, f, f}_\text{$3$-times},$ $\underbrace{f^{-1}, f^{-1}, f^{-1}}_\text{$3$-times},$ $\underbrace{f^{-1}, f^{-1}, f^{-1}, f^{-1}}_\text{$4$-times},\underbrace{f, f, f, f}_\text{$4$-times}...  \rbrace$ and $G=\lbrace f, f^{-1}, f^{-2}, f^{2}, f^{3}, f^{-3}, f^{-4}, f^{4}
\\
, . . . \rbrace$ are expansive with expansivity constant $ 0 < \alpha < \frac{1}{2}$.
\label{E2.3}
\end{Example}

This discussion allow us to study the following notions both of which implies expansivity.  

\begin{Definition}
(i) $F\in N(X)$ is said to be recurrently expansive if there exists $\mathfrak{c}\in (0,1)$ such that any distinct pair $x,y\in X$ satisfy $\limsup\limits_{n\rightarrow\infty}d(F_{n}(x), F_{n}(y)) > \mathfrak{c}$.
\\
(ii) $F\in N(X)$ is said to be mean expansive if there exists $\mathfrak{c}\in (0,1)$ such that any distinct pair $x,y\in X$ satisfy $\limsup\limits_{n\rightarrow \infty} \frac{1}{n}\sum_{i=0}^{n-1}d(F_{i}(x), F_{i}(y)) > \mathfrak{c}$.  
\label{D2.5}
\end{Definition}  

\begin{Proposition}
$F=\langle f\rangle$ is mean expansive implies it is recurrently expansive implies $f$ is positively expansive injective map. Consequently, there does not exist any autonomous mean expansive or recurrently expansive system on non-discrete compact metric space. 
\label{P3.3}
\end{Proposition}

\begin{proof} 
Clearly, if $F$ is recurrently expansive with expansive constant $\mathfrak{c}$, then $f$ is positively expansive with expansive constant $\mathfrak{c}$. Now, if $f(x) = f(y)$, then $f^{n}(x) = f^{n}(y)$ for all $n\in \mathbb{N}^{+}$ which implies that $\limsup\limits_{n\rightarrow \infty} d(f^{n}(x), f^{n}(y)) = 0$. Therefore, we must have $x=y$. 
\medskip

Conversely, suppose that $f$ is positively expansive injective map with expansive constant $\mathfrak{c}$. If $x\neq y$ in $X$, then by positive expansivity there exists $n_{1}\in \mathbb{N}$ such that $d(f^{n_{1}}(x), f^{n_{1}}(y))>\mathfrak{c}$. By injectivity and positive expansivity, we can choose $n_{2}\in \mathbb{N}^+$ such that $d(f^{n_{2}+n_{1}+1}(x), f^{n_{2}+n_{1}+1}(y)) > \mathfrak{c}$. Continuing in this way, we can choose a strictly increasing sequence $\lbrace m_{i}\rbrace_{i=1}^{\infty}$ such that $d(f^{m_{i}}(x), f^{m_{i}}(y))>\mathfrak{c}$ for all $i\in \mathbb{N}^{+}$. Therefore, we have $\limsup\limits_{n\rightarrow \infty}d(f^{n}(x), f^{n}(y))>\mathfrak{c}$ and hence, the result.   
\end{proof} 

\begin{Proposition}
If $F\in N(X)$ and $G\in N(Y)$, then $F$ and $G$ are recurrently expansive (resp. mean expansive) if and only if $F\times G$ is recurrently expansive (resp. mean expansive).
\label{P3.7} 
\end{Proposition} 

\begin{Proposition}
If $F\in N(X)$ is equicontinuous, then $F$ is recurrently expansive if and only if $F^{k}$ is recurrently expansive for all $k\in \mathbb{N}^{+}$.
\label{P3.8} 
\end{Proposition}
\begin{proof}
Proof is similar to the proof of Theorem 2.2 \cite{TDT}
\end{proof} 

\begin{Example} 
Let $X =\prod_{i\in \mathbb{Z}} X_{i}$ be equipped with the metric $d(x,y) = \sum_{i=-\infty}^{\infty}\frac{|x_{i} - y_{i}|}{2^{|i|}}$, where $X_{i} = \lbrace 0, 1\rbrace$. Let $f : X\rightarrow X$ be given by $f(x) = y$, where $y_{i} = x_{i+1}$ for all $i\in \mathbb{Z}$. Then, $F = \lbrace f, f^{-1}, f^{-2}, f^{2}, f, f^{-1}, f^{-2}, f^{2}, f^{3}, f^{-3}, f^{-4}, f^{4}, f, f^{-1}, f^{-2}, f^{2},$ $f^{3}, f^{-3}, f^{-4}, f^{4}, f^{5}, f^{-5}, f^{-6}, f^{6}, f^{7}, f^{-7}, f^{-8}, f^{8}, . . . \rbrace$ is recurrently expansive with expansive constant $0<\alpha <\frac{1}{2}$ but for each $i\in\mathbb{N}^+$, $F^{2i}$ is not recurrently expansive. Thus, we conclude that Proposition \ref{P3.8} is not be true if NAS is not equicontinuous.  
\label{E3.9}
\end{Example} 

\begin{Proposition}
Let $F\in N(X)$. If for some $k\in\mathbb{N}^+$, $F^{k}$ is mean expansive, then $F$ is mean expansive.  
\label{P3.10}
\end{Proposition}
\begin{proof}
Suppose that $F^{k}$ is mean expansive with expansive constant $\mathfrak{c}$. For $x, y\in X$ and $n\in \mathbb{N}^{+}$, we have 
$\frac{1}{n}\sum_{i=0}^{n-1}d((F^{k})_{i}(x), (F^{k})_{i}(y)) \leq  k\frac{1}{nk}\sum_{i=0}^{(n-1)k}d(F_{i}(x), F_{i}(y))\leq k \frac{1}{nk}\sum_{i=0}^{nk-1}d(F_{i}(x), F_{i}(y))$. From this we conclude that $F$ is mean expansive with expansive constant $\frac{\mathfrak{c}}{k}$.
\end{proof} 

We now give an example showing that the converse of the above result is not true. We further give sufficient condition under which the converse holds.

\begin{Example}
Let $f:\mathbb{R}\rightarrow \mathbb{R}$ is given by $f(x) = 2x$. Then, $F = \lbrace f, f^{-1}, f^{2}, f^{-2}, f^{3},
\\
 f^{-3}, . . .\rbrace$ is mean expansive but $F^{2i}$ is not mean expansive for each $i\in \mathbb{N}^{+}$. 
\label{E3.12}
\end{Example}   

\begin{Proposition}
Let $F\in N_{m}(X)$ be ME. If $F$ is mean expansive, then for each $k\in\mathbb{N}^+$, $F^{k}$ is mean expansive.  
\end{Proposition}
 
\begin{proof} 
Suppose that $F\in N_{m}(X)$ is mean expansive with expansive constant $\mathfrak{c}$. Fix $k\in \mathbb{N}^{+}$ and choose $\mathfrak{d} > 0$ such that for every pair of sequences $\lbrace x_{i}\rbrace_{i=0}^{\infty}$ and $\lbrace y_{i}\rbrace_{i=0}^{\infty}$, $\limsup\limits_{n\rightarrow \infty}\frac{1}{n}\sum_{i=0}^{n-1}d(x_{i}, y_{i}) < m\mathfrak{d}$ implies that $\limsup\limits_{n\rightarrow \infty}\frac{1}{n}\sum_{i=0}^{n-1}d(F_{j}(x_{i}), F_{j}(y_{i})) < \frac{\mathfrak{c}}{mk}$ for all $0\leq j\leq (mk - 1)$. 

If for $x,y\in X$, $\limsup\limits_{n\rightarrow \infty}\frac{1}{n}\sum_{i=0}^{n-1}d((F^{k})_{i}(x), (F^{k})_{i}(y)) < \mathfrak{d}$, then we must have $\limsup\limits_{n\rightarrow \infty}\frac{1}{n}\sum_{i=0}^{n-1}d(F_{mki}(x), F_{mki}(y)) < m\mathfrak{d}$. 

By ME, $\limsup\limits_{n\rightarrow \infty}\frac{1}{n}\sum_{i=0}^{n-1}d((F_{mki+j})(x), (F_{mki+j})(y)) < \frac{\mathfrak{c}}{mk}$ for all $0\leq j\leq (mk - 1)$ and hence 
\begin{align*}
\limsup\limits_{n\rightarrow \infty}\frac{1}{n}\sum_{i=0}^{n-1}d(F_{i}(x), F_{i}(y)) &\leq \limsup\limits_{n\rightarrow \infty}\frac{1}{n}\sum_{i=0}^{nmk-1}d(F_{i}(x), F_{i}(y)) \\
&\leq \sum_{j=0}^{mk-1}\limsup\limits_{n\rightarrow \infty}\frac{1}{n}\sum_{i=0}^{n-1}d(F_{mki+j}(x), F_{mki+j}(y)) \\
&< \mathfrak{c}
\end{align*}
By mean expansivity of $F$, we must have $x=y$. Hence, $F^{k}$ is mean expansive with expansive constant $\mathfrak{d}$. 
\end{proof}   

Let $\gamma=\lbrace x_{n}\rbrace_{n\in \mathbb{N}}$ be a sequence of elements of $X$. $\gamma$ is said to be $\delta$-pseudo orbit of $F$, if $d(f_{i+1}(x_{i}), x_{i+1}) < \delta$ for all $i\in \mathbb{N}$. $\gamma$ is said to be $\delta$-average-pseudo orbit of $f$ if there exists $N_{\delta}\in \mathbb{N}^{+}$ such that $\frac{1}{n}\sum_{i=0}^{n-1}d(f_{i+k+1}(x_{i+k}), x_{i+k+1}) < \delta$ for all $n\geq N_{\delta}$ and $k\in\mathbb{N}$. $\gamma$ is said to be $\epsilon$-shadowed by some $z\in X$, if $d(F_{n}(z), x_{n}) < \epsilon$ for all $n\in \mathbb{N}$. $\gamma$ is said to be $\epsilon$-shadowed in average by some $z\in X$, if $\limsup\limits_{n\rightarrow \infty} \frac{1}{n}\sum_{i=0}^{n-1}d(F_{i}(z), x_{i})$ $< \epsilon$. $\gamma$ is said to be almost $\epsilon$-shadowed by some $z\in X$ if $d(z, x_{0})< \epsilon$ and $\limsup\limits_{n\rightarrow \infty}d(F_{n}(z), x_{n}) < \epsilon$. $\gamma$ is said to be strongly $\epsilon$-shadowed in average if it is $\epsilon$-shadowed in average by some $z\in X$ such that $d(z, x_{0}) < \epsilon$. 
\medskip

$F\in N(X)$ is said to have shadowing property \cite{TDT} if for every $\epsilon>0$ there exists $\delta>0$ such that every $\delta$-pseudo orbit of $F$ can be $\epsilon$-shadowed by some point in $X$. 
\medskip

\begin{Definition}
(i) $F\in N(X)$ is said to have almost shadowing property (ALSP) if for every $\epsilon > 0$ there exists $\delta > 0$ such that every $\delta$-pseudo orbit of $F$ can be almost $\epsilon$-shadowed by some point in $X$.   
\\
(ii) $F\in N(X)$ is said to have strong average shadowing property (SASP) if for every $\epsilon > 0$ there exists $\delta > 0$ such that every $\delta$-average pseudo orbit of $F$ can be strongly $\epsilon$-shadowed in average by some point in $X$.  
\label{D2.7}
\end{Definition}

\begin{Proposition}
If $F\in N(X)$ is equicontinuous, then $F$ has ALSP if and only if $F^{k}$ has ALSP, for each $k>1$.  
\label{P3.13}
\end{Proposition}
\begin{proof}
One can prove similarly as the proof of Theorem 3.3 and Theorem 3.5 in \cite{TDT}. 
\end{proof}

\begin{Proposition}
Let $(X,d)$ and $(Y,p)$ be metric spaces and $F\in N(X)$, $G\in N(Y)$. Then, $F$ and $G$ has SASP (ALSP) if and only if $F\times G$ has SASP (ALSP). 
\label{P3.14} 
\end{Proposition}
\begin{proof}
Suppose that $F$ and $G$ has SASP. Let $\epsilon > 0$ and $\delta > 0$ be given for $\frac{\epsilon}{2}$ by strong average shadowing property of $F$ and $G$. Let $\lambda = \lbrace (x_{n}, y_{n})\rbrace_{n=0}^{\infty}$ be a $\delta$-average pseudo orbit of $F\times G$. Since $d(f_{n+1}(x_{n}), x_{n+1})\leq q((f_{n+1}\times g_{n+1})(x_{n}, y_{n}), (x_{n+1}, y_{n+1}))$ and $p(g_{n+1}(y_{n}), y_{n+1})\leq q((f_{n+1}\times g_{n+1})(x_{n}, y_{n}), (x_{n+1}, y_{n+1}))$, therefore $\gamma = \lbrace x_{n}\rbrace_{n=0}^{\infty}$ and $\eta = \lbrace y_{n}\rbrace_{n=0}^{\infty}$ are $\delta$-average pseudo orbits of $F$ and $G$ respectively. If $\gamma$ and $\eta$ are strongly $\frac{\epsilon}{2}$-shadowed in average by $x$ and $y$ through $F$ and $G$ respectively, then $\lambda$ is strongly $\epsilon$-shadowed in average by $(x,y)$ through $F\times G$. 
\medskip

Conversely, suppose that $F\times G$ has SASP. Let $\epsilon > 0$ and $\delta > 0$ be given for $\frac{\epsilon}{2}$ by strong average shadowing of $F\times G$. Let $\gamma = \lbrace x_{n}\rbrace_{n=0}^{\infty}$ be a $\delta$-average pseudo orbit of $F$.  For some $y\in Y$, let $\eta = \lbrace y_{n} = G_{n}(y) \rbrace_{n=0}^{\infty}$. Clearly, $\lambda = \lbrace (x_{n}, y_{n})\rbrace_{n=0}^{\infty}$ is a $\delta$-average pseudo orbit of $F\times G$ and hence, strongly $\epsilon$-shadowed in average by some point, say $(x, z)$. It is easy to see that, $\gamma$ is strongly $\epsilon$-shadowed in average by $x$ through $F$ and hence, $F$ has SASP. Similarly, one can prove that $G$ has SASP. 
\medskip

Proof of $F$ and $G$ has ALSP if and only if $F\times G$ has ALSP, is left to the reader. 
\end{proof}  

We say that $F$ is transitive if for every pair of non-empty open sets $U$ and $V$, there exists $n\in \mathbb{N}^{+}$ such that $F_{[i, i+(n-1)]}(U)\cap V\neq \phi$ for all $i\in \mathbb{N}^{+}$. 
\medskip

For $x,y\in X$, we write $x\mathsf{R}_{\delta} y$ if there exists $n\in \mathbb{N}^{+}$ such that for each $j\in \mathbb{N}^{+}$ there exists a finite sequence $x = x_{0}^{j},x_{1}^{j}, . . ., x_{n-1}^{j}, x_{n}^{j}=y$ such that $d(f_{j+i}(x_{i}^{j}), x_{i+1}^{j}) < \delta$ for all $0\leq i \leq (n-1)$. We write $x\mathcal{R}_{\delta} y$ if $x \mathsf{R}_{\delta} y$ and $y \mathsf{R}_{\delta} x$ and further, write $x\mathcal{R} y$ if $x\mathcal{R}_{\delta} y$ for all $\delta > 0$. We say that $F$ is chain transitive if $x\mathcal{R} y$ for all $x, y\in X$.
\medskip

If $F=\langle f\rangle$, then this definition boils down to the following in case of autonomous systems.  
\medskip

$f$ is said to be chain transitive if for every $\delta>0$ and $x,y\in X$, there exists a finite sequence $z_{0} = x, z_{1}, z_{2},..., z_{m+1} = y$ of elements in $X$ such that $d(f(z_{i}), z_{i+1}) < \delta$ for all $0\leq i <m$. 

\begin{theorem} 
If $F$ is equicontinuous transitive system, then it is chain transitive. 
\end{theorem}
\begin{proof}
For given $\epsilon > 0$, choose $\delta > 0$ by equicontinuity of $F$. By transitivity, choose $n\in \mathbb{N}^{+}$ such that $F_{[j, j+(n-1)]}(B(x, \delta))\cap B(y, \delta)\neq \phi$ for all $j\in \mathbb{N}^{+}$. Thus for each $j\in \mathbb{N}^{+}$, there exists $z^{j}\in B(x, \delta)$ so that the sequence $x = x_{0}^{j}, x_{1}^{j} = f_{j}(z^{j}), x_{2}^{j} = f_{j+1}\circ f_{j}(z^{j}), . . ., x_{n-1}^{j} = f_{j+(n-2)}\circ . . .\circ f_{j+1}\circ f_{j}(z^{j}), x_{n}^{j} = y$ satisfies $d(f_{j+i}(x_{i}^{j}), x_{i+1}^{j}) < \epsilon$ for all $0\leq i \leq (n-1)$ and hence, $x\mathsf{R}_{\epsilon} y$. Since $j, x, y$ and $\epsilon$ were chosen arbitrarily, we conclude that $F$ is chain transitive. 
\end{proof} 

\begin{theorem}
If $F$ is surjective chain transitive system with shadowing property, then $F$ is transitive.    
\end{theorem}
\begin{proof}
Let $x, y\in X$ and $\epsilon > 0$. Let $\delta > 0$ be given for $\epsilon$ by the shadowing property of $F$. By chain transitivity of $F$, there exists $n\in\mathbb{N}^+$ such that for each $j\in \mathbb{N}^{+}$ there exists a finite sequence $x = x_{0}^{j},x_{1}^{j}, . . ., x_{n-1}^{j}, x_{n}^{j}=y$ satisfying $d(f_{j+i}(x_{i}^{j}), x_{i+1}^{j}) < \delta$ for all $0\leq i \leq (n-1)$. Extend this sequence to a $\delta$-pseudo orbit $\eta = \lbrace z_{0},. . . , z_{j-3}, z_{j-2}, z_{j-1} = x = x_{0}^{j}, z_{j} = x_{1}^{j}, . . ., z_{j+(n-2)} = x_{n-1}^{j}, z_{j+(n-1)} = x_{n}^{j}= y, z_{j+n} = f_{j+n}(y),z_{j+(n+1)} = f_{j+(n+1)}(z_{j+n}), . . .\rbrace$, where $f_{i}(z_{i-1}) = z_{i}$ for $1\leq i\leq (j-1)$ and $z_{j+(n+k)} = f_{j+(n+k)}(z_{j+(n+k-1)})$ for all $k\geq 2$. By the shadowing property, there exists $w\in X$ such that $d(F_{n}(w), z_{n}) < \epsilon$ for all $n\in \mathbb{N}^{+}$. Therefore, $F_{j-1}(w) \in B(x, \epsilon)$ and $F_{[j, j+(n-1)]}(F_{j-1})(w) \in B(y, \epsilon)$. Hence, $F_{[j, j+(n-1)]}(B(x, \epsilon))\cap B(x, \delta) \neq \phi$. Since $j, x, y$ and $\epsilon$ were chosen arbitrarily, we conclude that $F_{[j, j+(n-1)]}(B(x, \epsilon))\cap B(x, \delta) \neq \phi$ for all $j\in \mathbb{N}^{+}$.  Hence the result.  
\end{proof}

The following two results show relations among shadowing property, almost shadowing property, average shadowing property and strong average shadowing property in case of an autonomous system. Unfortunately, we do not know much in case of a nonautonomous system.  

\begin{Lemma}
Let $F=\langle f\rangle$ on compact $X$ be chain transitive. Then, $F$ has ALSP if and only if it has shadowing property.  
\label{ET2.6}
\end{Lemma}
\begin{proof}
As the converse implication is automatic from the definition, we prove the forward implication. Suppose that $F$ has ALSP but does not have shadowing. For any given $\epsilon > 0$ and for each $n\in\mathbb{N}$, we can choose $\frac{1}{n}$-pseudo orbit $\alpha_{n}$ for $F$ which cannot be $\epsilon$-shadowed. Let $\delta>0$ be given for this $\epsilon$ by ALSP and fix $k\in \mathbb{N}^{+}$ such that $\frac{1}{k}<\delta$. By chain transitivity, choose finite $\frac{1}{m}$-pseudo orbits $\gamma_{m}$ for $f$ such that $\alpha_{m}\gamma_{m}\alpha_{m+1}$ forms a finite $\frac{1}{m}$-pseudo orbits for $f$, for all $m\geq k$. Clearly, $\alpha_{k}\gamma_{k}\alpha_{k+1}\gamma_{k+1}\alpha_{k+2}. . .$ forms a $\delta$-pseudo orbit for $f$ and hence, it can be almost $\epsilon$-shadowed. Therefore, there exists $p\in \mathbb{N}^{+}$ such that for all $j\geq p$, $\alpha_{j}$ is $\epsilon$-shadowed by some point in $X$, a contradiction. 
\end{proof}

\begin{theorem}
If $f$ is chain transitive continuous surjective map with ALSP on a compact metric space $X$, then $f$ has average shadowing property if and only if $f$ has SASP. 
\label{EC2.7}
\end{theorem} 
\begin{proof}
Since the converse implication is automatic, we prove the forward implication. By Theorem \ref{ET2.6}, $f$ has shadowing property. By Theorem 1 \cite{KOA}, $f$ has specification property and Lemma 12 \cite{KOA} implies that $f$ has SASP. 
\end{proof}

\begin{Example}
Let $X=\mathbb{R}$ be given with the usual metric and choose $m  > 1$. Define $g_{m} : X\rightarrow X$ by $g_m(x) = mx$ for all $x\in X$. Consider $F_{m} = \lbrace f_{i}\rbrace_{i\in \mathbb{N}^{+}}$ such that for every $i\in \mathbb{N}^{+}$, exactly one of the triplet $f_{3i+1}, f_{3i+2}, f_{3i+3}$ is $g_{m}$ and the other two are identity maps on $X$. Note that, $F_{m}$ need not be a periodic system. Also, $F_{m}$ is equicontinuous, mean equicontinuous, recurrently expansive and mean expansive. Since $F_{m}^{3} = \langle g_{m}\rangle$, by Proposition \ref{P3.13}, $F_m$ has ALSP. 
\label{4.13}
\end{Example} 

\section{Sufficient Conditions For Topological Stability}

For a metric space $(X,d)$, define the bounded metric $d_{1}$ by $d_{1}(x,y)=$min$\lbrace d(x,y), 1\rbrace$. Let $(C(X),\eta)$ be the space of all continuous self maps on $X$, where the metric $\eta$ is defined by $\eta(f, g) =$ sup$_{x\in X} d_{1}(f(x), g(x)$). We define a metric $\gamma$ on $N(X)$ as $\gamma(F, G) = $sup$_{i\in \mathbb{N}} \eta( f_i, g_i)$, where $F = \lbrace f_i\rbrace_{i\in \mathbb{N}^{+}}$ and $G = \lbrace g_i\rbrace_{i\in \mathbb{N}^{+}}$.   

\begin{Definition}
$F\in NC(X)$ is said to be topologically stable if for every $\epsilon > 0$ there is $ \delta > 0$ such that for any $G\in NC(X)$ satisfying $\gamma(F, G) < \delta$ there exists a continuous map $h: X\rightarrow X$ such that $f_{i}\circ h = h\circ g_{i}$ for all $i\in \mathbb{N}$ and $d(h(x), x) < \epsilon$ for all $x\in X$.
\label{D4.1}
\end{Definition}

Note that if $F\in NC(X)$, then this notion is stronger than the notion of topological stability used in \cite{TDT}.   

\begin{theorem}
Let $(X,d)$ and $(Y,p)$ be two metric spaces and $F\in NC(X)$, $H\in NC(Y)$. If $F$ and $G$ are uniformly conjugate, then $F$ is topologically stable if and only if $G$ is topologically stable. In other words, topological stability is a uniform dynamical property. 
\label{T4.2}
\end{theorem}

\begin{proof}
Suppose that $F$ is topologically stable. Let $j:Y\rightarrow X$ be a uniform conjugacy between $F$ and $H$ i.e. $f_{i}\circ j = j\circ h_{i}$ for all $i\in \mathbb{N}$. We want to show that $H$ is topologically stable. For $\epsilon\in (0,1)$, let $\beta\in (0,1)$ be given for $\epsilon$ by the uniform continuity of $j^{-1}$ i.e. $d(x,y)<\beta$ implies $p(j^{-1}(x),j^{-1}(y))<\epsilon$. Let $\alpha\in (0,1)$ be given for $\beta$ by the topological stability of $F$. Further, let $\delta \in (0,1)$ be given for $\alpha$ by the uniform continuity of $j$ i.e. $p(x,y)<\delta$ implies $d(j(x),j(y))<\alpha$. Let $G\in NC(Y)$ be such that $\gamma(H, G) < \delta$. Hence, sup$_{i\in \mathbb{N}}\eta^{Y}(h_{i}, g_{i}) < \delta$ implying $\eta(j^{-1}\circ f_{i}\circ j, g_{i}) < \delta$ for all $i\in \mathbb{N}$. This implies that $d_{1}(j^{-1}\circ f_{i}\circ j, g_{i}) < \delta$ for all $i\in \mathbb{N}$. By uniform continuity of $j$, we get that $d_{1}(f_{i}\circ j(y), j\circ g_{i}\circ j^{-1}(j(y))) < \alpha$ for all $i\in \mathbb{N}$ and all $y\in Y$. Set $G' = \lbrace g'_{i} =  j\circ g_{i}\circ j^{-1} \rbrace_{i\in \mathbb{N}^{+}}$. By topological stability of $F$, there exists a continuous map $k : X\rightarrow X$ such that $f_{i}\circ k = k\circ g'_{i}$ for all $i\in \mathbb{N}^{+}$ and $d(k(x), x) < \beta$ for all $x\in X$. If we set $k' = j^{-1}\circ k\circ j$, then $h_{i}\circ k' = h_{i}\circ j^{-1}\circ k\circ j = j^{-1}\circ f_{i}\circ k\circ j = j^{-1}\circ k\circ g'_{i}\circ j = j^{-1}\circ k\circ j\circ g_{i}\circ j^{-1}\circ j = k'\circ g_{i}$ for all $i\in \mathbb{N}$. Also, by uniform continuity of $j^{-1}$, we have $p(k'(y), y) < \epsilon$ for all $y\in Y$. Hence the result. 

A proof of the converse implication follows immediately from the fact that $j$ is an uniform equivalence. 
\end{proof}

We say that a metric space is Mandelkern locally compact \cite{M} if every bounded subset is contained in a compact set. Observe that, a metric space is Mandelkern locally compact if and only if every closed ball of finite radius is compact. From now onwards, we assume that $X$ is a Mandelkern locally compact metric space. Without loss of generality, we also assume that $0 < \epsilon, \delta, \alpha, \beta, \mathfrak{c}, \mathfrak{c}' < 1$.

\begin{theorem}
Let $F\in NC(X)$ be equicontinuous recurrently expansive with expansive constant $\mathfrak{c}$. If $F$ has ALSP, then $F$ is topologically stable. Moreover, for every $\epsilon\in (0,\frac{\mathfrak{c}}{3})$ there exists $\delta > 0$ such that if $G\in NC(X)$ satisfies $\gamma(F, G) < \delta$, then there exists a unique continuous map $h : X\rightarrow X$ such that $F_{n}\circ h = h\circ G_{n}$ for all $n\in \mathbb{N}$ and $d(h(x), x) < \epsilon$ for all $x\in X$. In addition, if $G$ is expansive with expansive constant $\mathfrak{c'}\geq 3\epsilon$, then the conjugating map $h$ is injective.  
\label{T4.3}
\end{theorem}

\begin{Lemma}
Let $F$ be recurrently expansive with expansive constant $\mathfrak{c}$ and let $F$ has ALSP. For $\epsilon\in (0,\frac{\mathfrak{c}}{3})$, let $\delta\in (0,\frac{\mathfrak{c}}{3})$ be given by ALSP of $F$. Then, every $\delta$-pseudo orbit of $F$ can be almost $\epsilon$-shadowed by exactly one point.  
\label{L4.4}
\end{Lemma}
\begin{proof}
Let $\lbrace x_{n}\rbrace_{n\in\mathbb{N}}$ be a $\delta$-pseudo orbit of $F$ and Suppose that $x,y\in X$ almost $\epsilon$-shadow a $\delta$-pseudo orbit $\lbrace x_n\rbrace_{n\in\mathbb{N}}$ of $F$. Since $d(F_{n}(x), F_{n}(y)) \leq d(F_{n}(x), x_{n}) + d(x_{n}, F_{n}(y))$ for all $n\in \mathbb{N}$, we have  
\begin{align*}
\limsup\limits_{n\rightarrow \infty}d(F_{n}(x), F_{n}(y)) &\leq \limsup\limits_{n\rightarrow \infty}(d(F_{n}(x), x_{n}) + d(x_{n}, F_{n}(y))) \\ 
&\leq \limsup\limits_{n\rightarrow \infty}d(F_{n}(x), x_{n}) +\limsup\limits_{n\rightarrow \infty}d(x_{n},F_{n}(y)) \\
&\leq 2\epsilon < \mathfrak{c}
\end{align*}
By recurrent expansivity of $F$, we get that $x=y$. Hence the result.  
\end{proof}

\begin{Lemma}
Let $F$ be recurrently expansive with expansive constant $\mathfrak{c}$. For any $x_{0}\in X$ and $\lambda >0$, there exists $N > 0$  such that $d(F_{n}(x_{0}), F_{n}(x)) \leq \mathfrak{c}$ for all $0\leq n\leq N$ implies  $d(x_{0}, x) < \lambda$.
\label{L4.5}
\end{Lemma} 
\begin{proof}
Choose a sequence $\lbrace x_{N}\rbrace_{N\in \mathbb{N}^{+}}$ in $X$ such that $d(F_{n}(x_{0}), F_{n}(x_{N})) \leq \mathfrak{c}$ for all $0\leq n \leq N$ and $d(x_{0}, x_{N}) \geq \lambda$. Since $B[x_{0}, \mathfrak{c}]$ is compact, we can assume that $x_{N}$ converges to $x$, for some $x\in X$. By continuity of each $F_n$, we have $d(F_{n}(x_{0}), F_{n}(x)) \leq \mathfrak{c}$ for all $n \in \mathbb{N}$ and $d(x_{0}, x) \geq \lambda$, a contradiction to the recurrent expansivity of $F$.
\end{proof}

\textbf{Proof of Theorem \ref{T4.3}} 
For $\epsilon\in (0,\frac{\mathfrak{c}}{3})$, choose $\beta\in (0,\epsilon)$ by the equicontinuity of $F$. Let $\delta\in (0,\beta)$ be given for $\beta$ by the ALSP of $F$. Let $G\in NC(X)$ be such that $\gamma(F, G) < \delta$ i.e. $\eta(f_{i}(x), g_{i}(x)) < \delta$ for all $i\in \mathbb{N}$ and all $x\in X$. Thus for all $x\in X$, $\lbrace G_{n}(x)\rbrace_{n\in \mathbb{N}}$ forms a $\delta$-pseudo orbit of $F$. By Lemma \ref{L4.4}, define $h : X\rightarrow X$, where $h(x)$ is a unique almost-$\beta$-tracing point of the $\delta$-pseudo orbit $\lbrace G_{n}(x)\rbrace_{n\in \mathbb{N}}$ i.e. $\limsup\limits_{n\rightarrow \infty} d(F_{n}(h(x)), G_{n}(x))<\beta$ for all $x\in X$, where $d(h(x), x) < \epsilon$. Note that, for all $x\in X$ and all $i\in \mathbb{N}$, we have
\begin{align*}
\limsup\limits_{n\rightarrow \infty}d(F_{n}(f_{i}(h(x))), F_{n}(h(g_{i}(x)))) &\leq \limsup\limits_{n\rightarrow \infty}d(f_{i}F_{n}(h(x)), f_{i}G_{n}(x)) \\
&+ \limsup\limits_{n\rightarrow \infty}d(f_{i}G_{n}(x), g_{i}G_{n}(x))\\
&+ \limsup\limits_{n\rightarrow \infty}d( G_{n}(g_{i}(x)), F_{n}h(g_{i}(x)))\\
&< 3\epsilon < \mathfrak{c}
\end{align*} 
Hence by the recurrent expansivity of $F$, $f_{i}\circ h(x) = h\circ g_{i}(x)$ for all $i\in \mathbb{N}$. Hence, $F_{n}\circ h = h\circ G_{n}$ for all $n\in \mathbb{N}$.
\medskip

Now we show that $h$ is a continuous map. Let $x_0\in X$ and $\lambda > 0$. By Lemma \ref{L4.5}, there exists $N > 0$ such that for any $y\in X$, $d(F_{n}h(x_{0}), F_{n}h(y)) \leq \mathfrak{c}$ for all $n\leq N$ implies $d(h(x_{0}), h(y))  < \lambda$. Choose $\alpha > 0$ such that $d(x_{0}, y) < \alpha$ implies $d(G_{n}(x_{0}), G_{n}(y)) < \frac{\mathfrak{c}}{3}$ for all $n \leq N$ and all $y\in X$. Therefore, $d(x_{0}, y) < \alpha$ implies that for all $n \leq N$ and all $y\in X$, 
\begin{align*} 
d(F_{n}h(x_{0}), F_{n}h(y)) = d(hG_{n}(x_{0}), hG_{n}(y)) &\leq d(hG_{n}(x_{0}), G_{n}(x_{0})) 
\\
&+ d(G_{n}(x_{0}), G_{n}(y)) 
\\
&+ d(G_{n}(y), hG_{n}(y)) 
\\
&< \epsilon + \frac{\mathfrak{c}}{3} + \epsilon < \mathfrak{c} 
\end{align*}
Thus for all $y\in X$, we get that $d(h(x_{0}), h(y)) < \lambda$, whenever $d(x_{0}, y) < \alpha$ i.e. $h$ is continuous at $x_{0}$. Hence, $h$ is a continuous map.
\medskip

Assume that there exists another continuous map $h':X \rightarrow X$ such that $F_{n}\circ h' = h'\circ F_{n}$ for all $n\in \mathbb{N}$ and $d(h'(x), x) < \epsilon$ for all $x\in X$. 
Thus for all $n\in \mathbb{N}$ and all $x\in X$ 
\begin{align*}
d(F_{n}(h(x)), F_{n}(h'(x)))&\leq d(F_{n}(h(x)), G_{n}(x)) + d(G_{n}(x), F_{n}(h'(x)))\\
&= d(h(G_{n}(x)), G_{n}(x)) + d(G_{n}(x), h'(G_{n}(x))) \\
&< 2\epsilon < \mathfrak{c}
\end{align*}

Hence by recurrent expansivity of $F$, we have $h(x) = h'(x)$ for all $x\in X$. 
\medskip

Now assume that $h(x) = h(y)$. Since for each $n\in\mathbb{N}$  
\begin{align*}
d(G_{n}(x), G_{n}(y))& \leq d(G_{n}(x), h(G_{n}(x))) + d(hG_{n}(x), hG_{n}(y)) \\
&\hspace*{0.5cm} +  d(h(G_{n}(y)), G_{n}(y))\\ 
&< \epsilon + 0 + \epsilon \\
&= 2\epsilon < \mathfrak{c'}
\end{align*}
therefore by the expansivity of $G$ we get that $x = y$. 

\begin{theorem}
Let $F\in NC(X)$ be equicontinuous and expansive with expansive constant $\mathfrak{c}$. If $F$ has shadowing property, then $F$ is topologically stable. Moreover, for every $\epsilon\in (0,\frac{\mathfrak{c}}{3})$ there is $\delta > 0$ such that if $G\in NC(X)$ satisfies $\gamma(F, G) < \delta$, then there is a unique continuous map $h : X\rightarrow X$ such that $F_{n}\circ h = h\circ G_{n}$ for all $n\in \mathbb{N}$ and $d(h(x), x) < \epsilon$ for all $x\in X$. In addition, if $G$ is expansive with expansive constant $\mathfrak{c'}\geq 3\epsilon$, then $h$ is injective.  
\label{T4.6}
\end{theorem}
\begin{proof}
Proof is similar to the proof of Theorem \ref{T4.3}
\end{proof}

\begin{theorem}
Let $F\in NC(X)$ be mean equicontinuous and mean expansive with expansive constant $\mathfrak{c}$. If $F$ has SASP, then $F$ is topologically stable. Moreover, for every $\epsilon\in (0,\frac{\mathfrak{c}}{3})$ there is $\delta > 0$ such that if $G\in NC(X)$ satisfies $\gamma(F, G) < \delta$, then there is a unique continuous map $h : X \rightarrow X$ such that $F_{n}\circ h=h\circ G_{n}$ for all $n\in \mathbb{N}$ and $d(h(x), x)< \epsilon$ for all $x\in X$. In addition, if $G$ is mean expansive with expansive constant $\mathfrak{c'}\geq 3\epsilon$, then $h$ is injective.  
\label{T4.7}
\end{theorem}

\begin{Lemma}
Let $F$ be mean expansive with expansive constant $\mathfrak{c}$. Let $\delta > 0$ be given for $\epsilon\in (0,\frac{\mathfrak{c}}{3})$ by SASP of $F$. Then every $\delta$-average pseudo orbit of $F$ can be strongly $\epsilon$-shadowed in average uniquely. 
\label{L4.8}
\end{Lemma}
\begin{proof}
Let $\lbrace x_{n}\rbrace_{n\in\mathbb{N}}$ be a $\delta$-average pseudo orbit for $F$ and let it be strongly $\epsilon$-shadowed in average by $x,y\in X$. Then $d(F_{n}(x), F_{n}(y)) \leq d(F_{n}(x), x_{n}) + d(x_{n}, F_{n}(y))$ for all $n\in \mathbb{N}$. This implies that
\begin{align*}
\limsup\limits_{n\rightarrow \infty}\frac{1}{n}\sum_{i=0}^{n-1}d(F_{i}(x), F_{i}(y)) &\leq \limsup\limits_{n\rightarrow \infty}\frac{1}{n}\sum_{i=0}^{n-1}(d(F_{i}(x), x_{i}) + d(x_{i}, F_{i}(y))) \\
&\leq \limsup\limits_{n\rightarrow \infty}\frac{1}{n}\sum_{i=0}^{n-1}d(F_{i}(x), x_{i}) +\limsup\limits_{n\rightarrow \infty}\frac{1}{n}\sum_{i=0}^{n-1}d(x_{i},F_{i}(y)) \\ 
&< 2\epsilon < \mathfrak{c}
\end{align*}
By mean expansivity of $F$, we get that $x = y$. 
\end{proof}

\begin{Lemma}
Let $F$ be mean expansive with expansive constant $\mathfrak{c}$. For any $x_{0}\in X$ and $\lambda >0$, there exists $N > 0$  such that if $\frac{1}{n}\sum_{i=0}^{n-1}d(F_{i}(x_{0}), F_{i}(x)) \leq \mathfrak{c}$ for all $n\leq N$, then $d(x_{0}, x) < \lambda$ for all $x\in X$. 
\label{L4.9}
\end{Lemma} 
\begin{proof}
Choose a sequence $\lbrace x_{N}\rbrace_{N\in \mathbb{N}^{+}}$ in $X$ such that $\frac{1}{n}\sum_{i=0}^{n-1}d(F_{i}(x_{0}), F_{i}(x_{N})) \leq \mathfrak{c}$ for all $n\leq N$ and $d(x_{0}, x_{N}) \geq \lambda$. Since $B[x_{0}, \mathfrak{c}]$ is compact, we can assume that $x_{N}$ converges to $x\in X$. For each $n\in \mathbb{N}$ and $\epsilon > 0$ there exists $N'\geq n$ such that $d(F_{i}(x), F_{i}(x_{N})) < \epsilon$ for all $0\leq i\leq n$ and all $N \geq N'$. For such $N'$, $\frac{1}{n}\sum_{i=0}^{n-1}d(F_{i}(x), F_{i}(x_{N'})) < \epsilon$ and hence,
\begin{align*}
\frac{1}{n}\sum_{i=0}^{n-1}d(F_{i}(x_{0}), F_{i}(x)) &\leq \frac{1}{n}\sum_{i=0}^{n-1}d(F_{i}(x_{0}), F_{i}(x_{N'}))+\frac{1}{n}\sum_{i=0}^{n-1}d(F_{i}(x_{N'}), F_{i}(x)) \leq \mathfrak{c} + \epsilon
\end{align*}
Since $n$ and $\epsilon$ was chosen arbitrary, we get that $\limsup\limits_{n\rightarrow \infty}\frac{1}{n}\sum_{i=0}^{n-1}d(F_{i}(x_{0}), F_{i}(x)) \leq \mathfrak{c}$ and $d(x_{0}, x) \geq \lambda$, a contradiction to the mean expansivity of $F$. 
\end{proof}

\textbf{Proof of Theorem \ref{T4.7}} 
Let $\beta\in (0,\epsilon)$ be given for $\epsilon \in (0,\frac{\mathfrak{c}}{3})$ by mean equicontinuity of $F$. For this $\beta$, choose $\delta\in (0,\beta)$ such that every $\delta$-average pseudo orbit of $F$ is strongly $\beta$-shadowed in average. Let $G\in NC(X)$ be such that $\gamma(F, G) < \delta$ i.e. $\eta(f_{i}(x), g_{i}(x)) < \delta$ for all $x\in X$ and all $i\in \mathbb{N}$. Thus for all $x\in X$, the sequence $\lbrace G_{n}(x)\rbrace_{n\in \mathbb{N}}$ forms a $\delta$-average pseudo orbit for $F$. Define a map $h : X\rightarrow X$, where $h(x)$ is a unique strongly $\beta$-tracing point in average of the $\delta$-average pseudo orbit $\lbrace G_{n}(x)\rbrace_{n\in \mathbb{N}}$ i.e. $\limsup\limits_{n\rightarrow \infty} \frac{1}{n}\sum_{i=0}^{n-1}d(F_{i}(h(x)), G_{i}(x)) < \beta$ for all $x\in X$ and $d(h(x), x) < \beta$ for all $x\in X$. Note that, for all $x\in X$ and all $m \in \mathbb{N}$,

\begin{align*}
\limsup\limits_{n\rightarrow \infty}\frac{1}{n}\sum_{i=0}^{n-1}d(F_{i}(f_{m}(h(x))), F_{i}(h(g_{m}(x)))) &\leq \limsup\limits_{n\rightarrow \infty}\frac{1}{n}\sum_{i=0}^{n-1}d(f_{m}F_{i}(h(x)), f_{m}G_{i}(x)) \\
&+ \limsup\limits_{n\rightarrow \infty}\frac{1}{n}\sum_{i=0}^{n-1}d(f_{m}G_{i}(x), g_{m}G_{i}(x)) \\
&+ \limsup\limits_{n\rightarrow \infty}\frac{1}{n}\sum_{i=0}^{n-1}d(G_{i}(g_{m}(x)), F_{i}(h(g_{m}(x)))) \\
&< 3\epsilon < \mathfrak{c}
\end{align*}

By mean expansivity of $F$, we get that $f_{m}\circ h=h\circ g_{m}$ for all $m\in \mathbb{N}$ and hence, we have $F_{n}\circ h = h\circ G_{n}$ for all $n\in \mathbb{N}$. 
\medskip

We now show that $h$ is continuous. Let $x_{0}\in X$ and $\lambda > 0$. By Lemma \ref{L4.9}, there exists $N > 0$ such that for any $y\in X$, $\frac{1}{n}\sum_{i=0}^{n-1}d(F_{i}(x_{0}), F_{i}(y)) \leq \mathfrak{c}$ for all $n\leq N$ implies $d(x_{0}, y)  < \lambda$. 
Choose $\alpha > 0$ such that $d(x_{0}, y) < \alpha$ implies $d(G_{n}(x_{0}), G_{n}(y)) < \frac{\mathfrak{c}}{3}$ for all $n \leq N$ and all $y\in X$. 
Therefore, $d(x_{0}, y) < \alpha$ implies that for all $n \leq N$ and all $y\in X$,
\begin{align*}
\frac{1}{n}\sum_{i=0}^{n-1}d(F_{i}h(x_{0}), F_{i}h(y)) &= \frac{1}{n}\sum_{i=0}^{n-1}d(hG_{i}  (x_{0}), hG_{i}(y))\\ 
&\leq \frac{1}{n}\sum_{i=0}^{n-1}d(hG_{i}(x_{0}), G_{i}(x_{0})) + \frac{1}{n}\sum_{i=0}^{n-1}d(G_{i}(x_{0}), G_{i}(y))\\ 
&\hspace*{0.5cm}+ \frac{1}{n}\sum_{i=0}^{n-1}d(G_{i}(y), hG_{i}(y))\\ 
&< \mathfrak{c} 
\end{align*}
Thus $d(h(x_{0}), h(y)) < \lambda$ whenever $d(x_{0}, y) < \alpha$, for all $y\in X$ i.e. $h$ is continuous at $x_{0}$. So, $h$ is a continuous map.
\medskip

Assume that there exists another continuous map $h':X \rightarrow X$ satisfying $F_{n}\circ h' = h'\circ G_{n}$ for all $n\in \mathbb{N}$ and $d(h'(x), x) < \epsilon$ for all $x\in X$. 
Thus for all $n\in \mathbb{N}^{+}$ and all $x\in X$, 
\begin{align*}
\frac{1}{n}\sum_{i=0}^{n-1}d(F_{i}(h(x)), F_{i}(h'(x))) &\leq \frac{1}{n}\sum_{i=0}^{n-1}d(F_{i}(h(x)), G_{i}(x)) +\frac{1}{n}\sum_{i=0}^{n-1}d(G_{i}(x), G_{i}(h'(x)))\\ 
&= \frac{1}{n}\sum_{i=0}^{n-1}d(h(G_{i}(x)), G_{i}(x)) + \frac{1}{n}\sum_{i=0}^{n-1}d(G_{i}(x), h'(G_{i}(x))) \\
&< 2\epsilon < \mathfrak{c} 
\end{align*}
Hence by the mean expansivity of $F$, $h(x) = h'(x)$ for all $x\in X$. 
\medskip

Now assume that $h(x) = h(y)$. Since for each $n\in\mathbb{N}$, 
\begin{align*}
d(G_{n}(x), G_{n}(y))\leq d(G_{n}(x), h(G_{n}(x))) + d(hG_{n}(x), hG_{n}(y))+d(h(G_{n}(y)), G_{n}(y))< \mathfrak{c'}
\end{align*}
therefore by the expansivity of $G$ we get that $x = y$.

\begin{Corollary}
Every commutative equicontinuous mean expansive NAS with ALSP on a Mandelkern locally compact metric space is topologically stable.
\label{C4.11} 
\end{Corollary}

\begin{Corollary}\cite[Theorem 4]{WO}
Every autonomous expansive system with shadowing property on a compact metric space is topologically stable.
\label{C4.12}
\end{Corollary}

\textbf{Acknowledgements:} The first author is supported by CSIR-Junior Research Fellowship (File No.-09/045(1558)/2018-EMR-I) of  Government of India.


\begin{thebibliography}{24}
\bibitem{A} D. V. Anosov, On a class of invariant sets of smooth dynamical systems, Proc. 5th Int. Conf. on Nonlin. Oscill., 2, Kiev (1970), 39--45.  
\bibitem{AHT} N. Aoki, K. Hiraide, Topological theory of dynamical systems: recent advances, Elsevier, (1994). 
\bibitem{B} A. Z. Bahabadi, Shadowing and average shadowing properties for iterated function systems, Georgian Math. J., 22 (2015), 179--184.     
\bibitem{BD} M. L. Blank, Discreteness and continuity in problems of chaotic dynamics, Amer. Math. Soc., (1997).  
\bibitem{CD} P. Cull, Difference equations as biological models, Scientiae Mathematicae Japonicae, 64 (2006), 217--234. 
\bibitem{CKE} E. M. Coven, M. Keane, Every compact metric space that supports a positively expansive homeomorphism is finite, IMS Lecture Notes Monogr. Ser., Dynamics \& Stochastics, 48 (2006), 304--305. 
\bibitem{GSX} R. Gu, Y. Sheng, Z. Xia, The average shadowing property and transitivity for continuous flows, Chaos, Solitons and Fractals, 23 (2005), 989-995.  
\bibitem{KOA} D. Kwietniak, P. Oprocha, A note on the average shadowing property for expansive maps, Topology Appl., 159 (2012), 19--27. 
\bibitem{M} M. Mandelkern, Metrization of the one-point compactification, Proc. Amer. Math. Soc., 107 (1989), 1111--1115.  
\bibitem{MRT} D. Marcon, F.B. Rodrigues, Topological properties of discrete non-autonomous dynamical systems; 2016 preprint.
\bibitem{TDT} D. Thakkar, R. Das, Topological stability of a sequence of maps on a compact metric space, Bull. Math. Sci., 4 (2014), 99--111.
\bibitem{UU} W. R. Utz, Unstable homeomorphisms, Proc. Amer. Math. Soc., 6 (1950), 769--774.   
\bibitem{W} P. Walters, Anosov diffeomorphisms are topologically stable, Topology 9 (1970), 71--78.   
\bibitem{WO} P. Walters, On the pseudo orbit tracing property and its relationship to stability, Springer, Berlin, Heidelberg, (1978), 231--244.    
\end{thebibliography}
\end{document}